\documentclass{article}
\pdfoutput=1
\title{Coinductive Invertibility in Higher Categories}
\author{Alex Rice}
\date{2020-08-24}

\usepackage[utf8]{inputenc}
\usepackage[T1]{fontenc}
\usepackage{csquotes}
\usepackage{lmodern}
\usepackage{verbatim} \usepackage{geometry} \usepackage{amsmath}
\usepackage{amssymb} \usepackage{amsfonts} \usepackage{mathtools}
\usepackage{amsthm}
\usepackage{tikz} \usepackage{hyperref}
\usepackage{listings} \usepackage{color}
\usepackage[colorinlistoftodos,obeyDraft]{todonotes}

\usepackage[nodayofweek]{datetime}
\newdate{date}{24}{08}{2020}
\date{\displaydate{date}}

\usetikzlibrary{positioning}

\usepackage[ style=numeric, citestyle=numeric,
maxbibnames=99, url=false ]{biblatex}

\DefineBibliographyStrings{english}{%
  bibliography = {References}, }

\usepackage[capitalize]{cleveref}

\newtheorem{theorem}{Theorem} 
\newtheorem{cor}[theorem]{Corollary}
\newtheorem{lemma}[theorem]{Lemma} \theoremstyle{definition}
\newtheorem{definition}[theorem]{Definition} \theoremstyle{remark}
\newtheorem{remark}[theorem]{Remark} 

\DeclareMathOperator{\id}{id}

\newcommand{\linv}[1]{{}^\star\!#1} \newcommand{\rinv}[1]{#1^\star}
\newcommand{\inv}[1]{#1^{-1}} \newcommand{\comp}{\star}

\tikzset{
  every node/.style={
    rounded corners=5pt,
    fill=white
  }
}

\begin{document}
\maketitle

\listoftodos{}

\begin{abstract}
  Invertibility is an important concept in category theory. In higher
  category theory, it becomes less obvious what the correct notion of
  invertibility is, as extra coherence conditions can become necessary
  for invertible structures to have desirable properties.

  We define some properties we expect to hold in any reasonable definition of a
  weak \(\omega\)-category. With these properties we define
  three notions of invertibility inspired by homotopy type theory. These are
  \emph{quasi-invertibility}, where a two sided inverse is required,
  \emph{bi-invertibility}, where a separate left and right inverse is
  given, and \emph{half-adjoint inverse}, which is a quasi-inverse
  with an extra coherence condition. These definitions take the form
  of coinductive data structures. Using coinductive proofs we are able
  to show that these three notions are all equivalent in that given any
  one of these invertibility structures, the others can be obtained.
  The methods used to do this are generic and it is expected that the
  results should be applicable to any reasonable model of higher
  category theory.

  Many of the results of the paper have been formalised in Agda using
  coinductive records and the machinery of sized types.
\end{abstract}
\section{Introduction}\label{sec:intro}

In the study of higher category theory, the notion of invertibility is
central. It is common within higher category theory to not specify
that two objects are equal, or that two sides of an equation are
equal, but rather specify that there is a higher level equivalence
between these two objects. It is usual to say that two objects are
equivalent exactly when there is an invertible morphism between them.

This idea can be seen even in the simplest examples of higher
categories. Consider the definition of a monoidal category. It is our
intention that given objects \(A\), \(B\) and \(C\) that \(A \otimes
(B \otimes C)\) and \((A \otimes B) \otimes C\) represent the same
object. However it is often the case that requiring equality between
these objects (as in a strict monoidal category) is too restrictive
and so the more general definition states that there is a natural isomorphism
between them.

This is our first example of invertibility. An isomorphism between
\(A\) and \(B\) is a morphism \(f : A \to B\) such that there exists a
morphisms \(g : B \to A\) with \(f \circ g = \id_B\) and \(g \circ f =
\id_A\). This notion of isomorphism works well throughout standard
category theory and even in monoidal categories.

However if higher categories are considered, it can be seen that the
problem has just been pushed into the higher structure of the
category. We have specified that \(g \circ f\) and \(\id_A\) should be
equal, yet have already stated that equality is not the best notion of
two objects in a higher category being the same. Ideally we want a collection of data
that specifies that \(f\) has an inverse \(g\), and that there is a
higher morphism \(\alpha : g \circ f \to \id_A\) and an inverse
\(\beta : \id_A \to g \circ f\), and that there exists a morphism
\(\alpha \circ \beta \to \id_{\id_A}\) and continuing in this fashion.
However, this generates an unwieldy set of data that is difficult to
describe.

A natural way we can talk about infinite data structures like this is through
coinduction. The definition of equivalence above can be written nicely
by saying that \(f : A \to B\) is an equivalence if there exists a \(g
: B \to A\) and equivalences \(f \circ g \to \id_B\) and \(g \circ f
\to \id_A\). This is a neat method of defining an equivalence, but
also has the advantage of opening up more proof techniques to us. In
particular, it allows us to structure proofs coinductively.

Given this coinductive machinery, the definition of equivalence above
seems very natural. If we start with our definition of isomorphism,
then realise that the equalities introduced should themselves be
weakened to equivalences, we end up with the proposed definition.
However this is not the only way that invertibility could have been
defined. In type theory and in particular in HoTT~\cite{hottbook},
there are three notions of invertibility defined:

\begin{enumerate}
\item a \emph{quasi-inverse}, which is similar to the notion of
  isomorphism presented already;
\item a \emph{bi-inverse}, where instead of specifying an inverse to
  the function \(f\), we instead present separate left and right
  inverses;
\item a \emph{half-adjoint inverse}, defined to be the same as the
  quasi-inverse but with an extra coherence condition.
\end{enumerate}

There are certain properties of these invertibility conditions that
are important in type theory. In the language of HoTT, all three of
these definitions are equivalent~\cite[Corollary 4.3.3]{hottbook}, in
that a function \(f\) has one of these three types of inverses, then
it also has the other two. This can be viewed as the \((\infty,1)\)
case of what is proved in this paper, as all equalities are naturally
invertible in HoTT. Secondly, definitions 2 and 3 are
propositions~\cite[Theorem 4.2.13, Theorem 4.3.2]{hottbook}. This
means that there is at most one way in which they hold up to
propositional equality.

Although it seems that these three are not applicable to higher
categories due to the use of equality in their definitions, they can
nonetheless be modified into coinductive definitions suitable for such
a category.

\subsection{Contributions and further work}\label{sec:conclusion}

This paper explores the equivalence of various types of invertible
cells, using coinductive proof techniques. The aim is that these
techniques make managing the data of a cell being invertible easier
and the proofs simpler.

Given a setting of higher category theory which we believe is
sufficiently general, we have managed to prove many properties of
bi-invertible structures, which we believe have not been studied in
great depth before. We have shown that, using coinductive techniques,
one can work with these structures using only 2-dimensional reasoning.
Further we have shown that these structures work nicely with inverses
and composition as expected.

These results have been used to show that the three definitions of
invertibility given in the paper all imply each other, in that given
one of these structures we can construct the others. This mirrors
similar proofs in lower dimensional category theory that were already
known well.

Going further, we would like to be able to prove some of the
contractibility results that have been conjectured. It is possible
that this either requires a lot more structure than we have postulated
that any weak higher category should have. Nevertheless, we believe
that contractibility results for bi-invertible structures could have
large implications on ways in which higher categories can be defined
and worked with. It would also have implications on the study of
\(\infty\)-groupoids, where every morphism is invertible.

Lastly, this paper also contributes a large amount of Agda
formalisation for all the structures mentioned. With a little work,
much of the code could be reused for proving similar facts about
\(\omega\)-categories. The power of this formalisation technique would
be largely increased if the condition of ``respecting the graphical
calculus'' could be removed, which is hard to express in Agda. We believe that
just as the correctness of string diagrams for bi-categories follows
from the existence of associators, unitors and some coherence
conditions, there should be a similar set of conditions on
\(\omega\)-category which at least allows up to 2-dimensional
graphical reasoning. The ability to automate a translation from the
graphical calculus to a proof system such as Agda would allow very
complex proofs, including all the proofs in this paper, to be
formalised with ease.

\subsection{Related Work}\label{sec:related}

A lot of work has been done relating to invertible structures in
higher categories. The idea of defining a notion of pre-category is
largely inspired by \citeauthor{Cheng2007}~\cite{Cheng2007}.
\citeauthor{kansangian2009weak}~\cite{kansangian2009weak} discuss the
equivalence of some forms of weak inverse in strict \(n\)-categories.
The result that an equivalence implies an adjoint equivalence in
bicategories is well known and this has been extended to tricategories
by
\citeauthor{gurski2011biequivalences}~\cite{gurski2011biequivalences}.
\citeauthor*{Lafont_2010}~\cite{Lafont_2010} construct the idea
of a quasi-invertible structure in precisely the same way as in this
paper, and proves some of the same results, though these are done in
the setting of a strict \(\omega\)-category, where some of the
coinductive arguments become easier.

While the idea of viewing higher categories as coinductive data
structures is not
new~\cite{cheng2012weak,hirschowitz_et_al:LIPIcs:2015:5166}, we
believe the set up used in this paper is novel. The work of
\citeauthor{hirschowitz_et_al:LIPIcs:2015:5166} is similar enough that
it was possible to adapt it to form the basis for the
formalisation~\cite{rice_agda} in \Cref{sec:formalisation}.

\subsection{Formalisation}\label{sec:formalisation}

It may at first appear that a lot of the coinductive proofs in this
paper work by ``magic'' in that they seem like non-well founded
induction. To remedy this all the definitions and results from
~\cref{sec:categories,sec:bi-invertibility} have
been formalised in the proof assistant
Agda\footnote{\url{https://github.com/agda/agda}} with use of the
standard library\footnote{\url{https://github.com/agda/agda-stdlib}}.
The formalisation uses sized types~\cite{Abel_2010} to ensure
productivity of the functions defined and makes heavy use of
coinductive record types and
copatterns~\cite{Abel_2013}.

The notion of higher pre-category introduced in this paper is very
set-theoretic and becomes messy in the world of type theory. Because
of this, alternative coinductive definitions of globular sets and
composition are taken
from~\citetitle{hirschowitz_et_al:LIPIcs:2015:5166}~\cite{hirschowitz_et_al:LIPIcs:2015:5166}.
This allows a lot of the proofs to become neater as coinduction tends
to work nicely with coinductively defined structures. We believe the
use of a different definition is not consequential as the main purpose
of the formalisation was to show that coinductive elements of the
arguments hold.

The code is publicly available~\cite{rice_agda} and has been tested
with Agda version 2.6.1 and standard library version
1.3.

\subsection{Background}\label{sec:background}

This background section will give a quick introduction to coinduction and some of the ideas behind \(\omega\)-categories required for the paper as well as introducing the first notion on invertibility.

\subsubsection{Coinduction}\label{sec:coinduction}

Throughout the paper, we want to use coinductive data structures and
many functions will be generated by corecursion, and many proofs made
using coinduction. Like inductive data structures, coinductive data
structures are able to reference themselves in their definition. Many
of the coinductive results in this paper are proved in a style similar
to that given in \citetitle{kozen_silva_2017}~\cite{kozen_silva_2017},
and are justified by \emph{guardedness} which is a well-known
concept~\cite{Coquand_1994,Gim_nez_1995}.
Below, a summary of coinduction is given, all of which is standard.

Categorically, coinduction forms the categorical dual of induction.
Whereas inductive data structures can be represented as initial
algebras, coinductive data structures may be represented as terminal
coalgebras~\cite{jacobs1997tutorial}. To see the differences clearly
between these two constructions and how they relate, it will be
helpful to go through an example. Since a lot of the language used
throughout this paper is type-theoretic, induction and coinduction
will be explained in the same way.

When a type is introduced in type theory, there are four parts of its
definition: type formation, term formation, a recursion principle, and
an induction principle. These will be demonstrated with the example of
lists.
\begin{itemize}
\item Type formation tells us how to build the type. For lists, given
  a type \(A\), the type \(\mathsf{List}(A)\) can be formed.
\item Term formation tells us how terms of the type can be
  constructed. \(\mathsf{List}(A)\) has multiple constructors. We have
  \(\mathsf{nil} : \mathsf{List}(A)\), which is a constructor with no
  arguments and \(\mathsf{cons} : A \to \mathsf{List}(A) \to
  \mathsf{List}(A)\), which takes two.
\item Recursion lets us form functions from an inductive type. For
  lists this says that to form a function \(f : \mathsf{List}(A) \to
  B\), it suffices to define \(f(\mathsf{nil}) : B\) and define
  \(f(\mathsf{cons}(a)(xs))\) given the value of \(f(xs)\).
\item Induction is a generalisation of recursion and allows us to
  define functions \(f : \Pi_{xs : \mathsf{List}(A)} P(xs)\) where
  \(P\) is a type family of type \(\mathsf{List}(A) \to \mathcal{U}\),
  with \(\mathcal{U}\) being a universe of types. In other words \(P\)
  assigns a type to each list.
\end{itemize}

A further thing to note is that we will require the use of a
generalisation of this known as \emph{inductive type
  families}~\cite[Section 5.7]{hottbook}. These allow our inductive
types to be parameterised. A good example of this is the vector type,
which is a list of specified length. Here our type formation rule says
that given a type \(A\) we can get a type \(\mathsf{Vec}(A) :
\mathbb{N} \to \mathcal{U}\). Further the term formation rules become
\(\mathsf{nil} : \mathsf{Vec}(A)(0)\) and \(\mathsf{cons} : \Pi_{n :
  \mathbb{N}} A \to \mathsf{Vec}(A)(n) \to
\mathsf{Vec}(A)(n + 1)\). Notice that unlike with lists,
which also took a parameter, vectors can not be defined by considering
each element of the parameter separately. The recursion and induction
rules also need to modified.

Now it will be seen how this differs for coinductive types. For this
we will use one of the most common coinductive types, the stream,
which are infinite lists.

\begin{itemize}
\item Type formation remains the same as before. Given any type \(A\)
  we can form \(\mathsf{Stream}(A)\).
\item Instead of specifying how to form terms, we specify how to
  deconstruct terms. Streams have two deconstructors: \(\mathsf{head}
  : \mathsf{Stream}(A) \to A\) and \(\mathsf{tail} :
  \mathsf{Stream}(A) \to \mathsf{Stream}(A)\).
\item Instead of recursion, we have corecursion, which specifies how
  to build functions \(f : B \to \mathsf{Stream}(A)\). Whereas in recursion, we
  specified the behaviour of the function on each constructor, in
  corecursion we specify how to deconstruct the value returned.
  Therefore it suffices to provide \(\mathsf{head}(f(b)) : A\) and to
  provide \(\mathsf{tail}(f(b)) : \mathsf{Stream}(A)\), where \(\mathsf{tail}(f(b))\) can be defined by either by providing a stream or by providing a \(c : B\) (which may or may not equal \(b\)) such that \(\mathsf{tail}(f(b)) = f(c)\).
\item Coinduction does not work quite the same as induction, as
  flipping the arrow no longer works due to the dependency in the
  type. Instead, if we are trying to prove a predicate \(P :
  \mathsf{Stream}(A) \to \mathcal{U}\), we can define the predicate
  coinductively, and then provide a witness to it by corecursion.
  Coinductive type families will likely be needed to define the
  predicate.
\end{itemize}

As an example of coinduction, suppose we wanted to prove every element
of a stream was even. First we construct the predicate \(P :
\mathsf{Stream}(\mathbb{N}) \to \mathcal{U}\) as having deconstructors
\(\mathsf{headProof}(xs) = \mathsf{is{\text
    -}even}(\mathsf{head}(xs))\) and \(\mathsf{tailProof}(xs) =
P(\mathsf{tail}(xs))\). An element of \(P(xs)\) (a proof that every
element of the stream is even) could then be formed by corecursion.

\subsubsection{\texorpdfstring{\(\omega\)}{Omega}-Categories}\label{sec:categories}

In this paper we want to talk about different structures on
\(\omega\)-categories, however there are many different definitions of
\(\omega\)-categories~\cite{leinster2001survey}. Therefore, similar in style to a
paper by \textcite{Cheng2007}, we give a set of conditions that we
expect to hold in any reasonable definition of a weak infinity
category. Results are then proved using only these conditions, and
we reason that given a precise definition of an \(\omega\)-category, if
one could show that these conditions hold, the results should
follow for free.

\begin{definition}
  The \emph{globe category} \(\mathbb{G}\) is the category where the
  objects are the natural numbers and morphisms are generated by
  \begin{align*}
    &\sigma_n : n \to n+1\\
    &\tau_n : n \to n+1
  \end{align*}
  subject to the conditions
  \begin{align*}
    \sigma_{n + 1} \circ \sigma_n &= \tau_{n + 1} \circ \sigma_n\\
    \sigma_{n + 1} \circ \tau_n &= \tau_{n + 1} \circ \tau_n
  \end{align*}
\end{definition}

\begin{definition}
  A \emph{globular set} \(G\) is a presheaf on \(\mathbb{G}\). We
  refer to \(G(n)\) as the set of \(n\)-cells and for two \(n\)-cells
  \(x\) and \(y\) we write \(f: x \to y\) to mean that \(f\) is an
  \((n+1)\)-cell and
  \begin{align*}
    G(\sigma_n)(f) &= x\\
    G(\tau_n)(f) &= y
  \end{align*}
  where we call \(x\) the source of \(f\) and \(y\) the target of
  \(f\).
\end{definition}

\begin{definition}
  A \emph{globular set with identities and composition} is a globular
  set \(G\) with the following:
  \begin{itemize}
  \item For each \(n\)-cell \(x\), there is an \((n+1)\)-cell, \(\id_x
    : x \to x\).
  \item Inductively define composition as follows:
    \begin{itemize}
    \item Given \(n\)-cells \(f: x \to y\) and \(g: y \to z\) there is
      an \(n\)-cell \(g \comp_0 f: x \to z\).
    \item Given \(\alpha: f \to g\) and \(\beta: h \to j\), where the
      composites \(f \comp_n h\) and \(g \comp_n j\) are well-defined,
      there is a morphism \(\alpha \comp_{n+1} \beta: (f \comp_n h)
      \to (g \comp_n j)\)
    \end{itemize}
  \end{itemize}
\end{definition}

Once we have a globular set with identities and composition, we have
enough to define a notion of equivalence. The most basic notion of
invertibility will be given here, as it will be needed to state the
remainder of properties that we expect an higher category to obey.

\begin{definition}
  Given a globular set \(G\) with identities and composition, with an
  \(n\)-cell \(f : x \to y\), a \emph{quasi-invertible} structure on
  \(f\) is a tuple \((\inv f, f_R, f_L, f_R^{QI}, f_L^{QI})\) where:
  \begin{itemize}
  \item \(\inv f\) is an \(n\)-cell \(y \to x\);
  \item \(f_R\) is an \((n+1)\)-cell \(f \comp_0 \inv f \to \id_y\);
  \item \(f_L\) is an \((n+1)\)-cell \(\inv f \comp_0 f \to \id_x\).
  \item \(f_R^{QI}\) is a quasi-invertible structure on \(f_R\).
  \item \(f_L^{QI}\) is a quasi-invertible structure on \(f_L\).
  \end{itemize}
\end{definition}

\begin{remark}
  The previous definition is a coinductive one. Formally, we have
  defined a coinductive data type, which we could call
  \(\mathsf{QuasiInvertible}(f)\), which references itself (by saying
  that \(f_R\) and \(f_L\) themselves have a quasi-invertible
  structure). Note that as \(\mathsf{QuasiInvertible}\) is dependent
  on the parameter \(f\), it is in fact a coinductive type family as
  described in \Cref{sec:coinduction}.

  Further, it should be noted that, as a coinductive structure,
  \(\mathsf{QuasiInvertible}(f)\) contains an infinite stack of data,
  and in particular contains the data witnessing the invertibility of
  \(f\) at all dimensions.
\end{remark}

\begin{remark}\label{descendants}
  Notice that in a globular set \(G\) with identities and composition,
  given two \(n\)-cells \(x\) and \(y\) with the same source and
  target, a new globular set \(G_{x,y}\) can be defined:
  \begin{itemize}
  \item the \(0\)-cells are the \((n+1)\)-cells in \(G\) with source
    \(x\) and target \(y\);
  \item the \(m\)-cells are the \((n+m+1)\)-cells in \(G\) whose
    source and target lie in the \((m-1)\)-cells of \(G_{x,y}\).
  \end{itemize}
  It is clear that we can carry over identities and composition to
  this new globular set.
\end{remark}

Consider the standard string diagram diagrammatic calculus for
bicategories~\cite{Heunen_2019,selinger2010survey}. We
can draw diagrams containing nodes, lines, and areas where areas
represent \(0\)-cells, a line between areas representing cells \(x\)
and \(y\) represents a \(1\)-cell between \(x\) and \(y\), and node
between lines represent \(2\)-cells between the cells which those lines
represent. These diagrams still represent well formed morphisms in an
globular set with identities and composition (at least in the presence
of unitors and associators, which we introduce below). In a bicategory
there exists the theorem that given a planar isotopy of string
diagrams, the source and target morphisms are equal. In this paper, string diagrams are written bottom to top and right to left, like in \cref{fig:example-string}.

\begin{figure}
  \centering
  \begin{tikzpicture}
    \node (Base1) {\(f\)};
    \node [on grid, draw, above=20pt of Base1] (A) {\(\alpha\)};
    \node [on grid, above left=40pt and 15pt of A] (Top1) {\(g\)};
    \node [on grid, draw, above right=20pt and 15pt of A] (B) {\(\beta\)};
    \draw (Base1) to (A);
    \draw (A) to[out=180,in=-90,looseness=0.9] ++(-15pt,20pt) to (Top1);
    \draw (A) to[out=0,in=-90,looseness=0.9] (B);
    \matrix [right=20pt] at (current bounding box.east) {
      \node [] {\(f : x \to y\)};\\
      \node [] {\(g : x \to y\)};\\
      \node [] {\(h : y \to y\)};\\
      \node [] {\(\alpha : f \to g \comp_0 h\)};\\
      \node [] {\(\beta : h \to \id_y\)};\\
    };
  \end{tikzpicture}

  \caption{The morphism \(\alpha \comp_0 (\id \comp_1 \beta)\)}
  \label{fig:example-string}
\end{figure}

\begin{definition}\label{def:higher-cat}
  Say that a globular set \emph{respects the graphical calculus} if
  for any well typed string diagram and any planar isotopy of this
  diagram, there exists a \(3\)-cell from the source of the isotopy to
  the target. Further, it is required that there is a quasi-invertible
  structure on this morphism.

  We define a \emph{higher pre-category} to be a globular set with
  identities and composition such that any globular set
  generated as in \Cref{descendants} respects the graphical calculus.
  Further we assume the existence of the following morphisms:
  \begin{itemize}
  \item For \(n>0\) and each \(n\)-cell \(f: x \to y\), there are
    \(n+1\)-cells, known as unitors, \(\lambda_f: \id_y \comp_0 f \to
    f\) and \(\rho_f: f \comp_0 \id_x \to f\).
  \item Given \(f,g,h\), \(n>1\)-cells with suitable composition
    defined, we have an associator \(a_{f,g,h} : (f \comp_0 g) \comp_0
    h \to f \comp_0 (g \comp_0 h)\).
  \item For compatible morphisms \(f,g,h,j\), we have an interchanger
    \(i_{f,g,h,j} : (f \comp_n g) \comp_0 (h \comp_n j) \to (f \comp_0 h)
    \comp_n (g \comp_0 j)\).
  \item For suitable \(f,g\) and \(n\), there is a cell \(\id_f
    \comp_{n+1} \id_g \to \id_{f \comp_n g}\).
  \end{itemize}
  All these further morphisms must be equipped with a quasi-invertible
  structure.
\end{definition}

\begin{remark}
  It is worth stressing that it is not intended that a higher
  pre-category is in any way a definition of an \(\omega\)-category.
  Instead, this is just a list of properties that are expected to hold
  in any realistic definition of a higher category. The definition is
  made to be as general as possible so that the results can be as
  applicable as possible. Further, while this definition seems to imply that
  we require our definition of \(\omega\)-category to be enriched over
  \(\omega\)-categories, this is not the case. If \(G\) is a globular
  set then \(G_{x,y}\) will also be a globular set with no extra
  conditions. From here it makes sense to separately require that each
  of these globular sets respects the graphical calculus.
\end{remark}

To end this section we give two small lemmas, which demonstrates the
use of these higher pre-categories and will also be invaluable later.

\begin{lemma}\label{identity}
  Any identity morphism in a higher pre-category \(G\) has a
  quasi-invertible structure.
\end{lemma}

\begin{proof}
  Let \(x\) be any cell. Take \(\inv {\id_x} = \id_x\) and
  \({(\id_x)}_R = {(\id_x)}_L = \lambda_{id_x}\). By assumption,
  \(\lambda_{id_x}\) has a quasi-invertible structure which can be
  used to form a structure on \(\id_x\).
\end{proof}

\begin{lemma}\label{inverse-invert}
  Suppose \(f\) has a quasi-invertible structure. This induces a
  quasi-invertible structure on \(\inv f\).
\end{lemma}

\begin{proof}
  Let \((\inv f, f_R, f_L, f_R^{QI}, f_L^{QI})\) be a quasi-invertible
  structure on \(f\). Then \((f , f_L, f_R, f_L^{QI}, f_R^{QI})\) is a
  quasi-invertible structure on \(\inv f\).
\end{proof}

\section{Types of Invertibility}\label{sec:invertibility}

In this section we introduce two more types of invertibility. Some
lemmas are proved that help us to work with them and finally all three
types of invertibility are shown to imply each other. Throughout this
section it will be assumed that we are working in a higher
pre-category \(G\).

\subsection{Bi-Invertibility}\label{sec:bi-invertibility}

Normally throughout mathematics, when saying a map is invertible a
single inverse is specified, and it is shown that this cancels the
function when composed on the left and the right. Instead we can
consider a function that has a left inverse and right inverse,
an inverse only cancels the function on when composed on the
left (or right respectively). It usually does not make sense to
consider functions which have both a separate left and right inverse
as in most scenarios it can be proved that these are equal and any
computations will be simplified by treating them as the same.

However, the concept of bi-invertibility, having both a left and right
inverse, plays a role in type theory. It can be shown (perhaps not
surprisingly) that a morphism being bi-invertible implies that it is
invertible in the usual sense. More surprisingly, the data for being
bi-invertible is in some ways simpler than the data for being
invertible, in that it can be shown that any two ways of showing a
function is bi-invertible turn out to be equal.

Before these can be studied, we must define what we mean for a
morphism in a higher category to be bi-invertible, as a generalisation
of the idea in type theory.

\begin{definition}
  Given a globular set \(G\) with identities and composition, with an
  \(n > 0\) cell \(f : x \to y\), a \emph{bi-invertible} structure on
  \(f\) is a tuple \((\rinv f, \linv f, f_R, f_L, f_R^{BI}, f_L^{BI})\)
  where:
  \begin{itemize}
  \item \(\rinv f\) is an \(n\)-cell \(y \to x\);
  \item \(\linv f\) is an \(n\)-cell \(y \to x\);
  \item \(f_R\) is an \((n+1)\)-cell \(f \comp_0 \rinv f \to \id_y\);
  \item \(f_L\) is an \((n+1)\)-cell \(\linv f \comp_0 f \to \id_x\).
  \item \(f_R^{BI}\) is a bi-invertible structure on \(f_R\).
  \item \(f_L^{BI}\) is a bi-invertible structure on \(f_L\).
  \end{itemize}
\end{definition}

Next it is shown that invertibility implies bi-invertibility. This is
fairly trivial though it will be written out in full to demonstrate proof by
coinduction.

\begin{lemma}\label{inv-to-bi-inv}
  Let \(f : x \to y\) be invertible. Then \(f\) is bi-invertible.
\end{lemma}

\begin{proof}
  This is proven by constructing a corecursive function
  \(\mathsf{invToBiInv} : \Pi_{f : x \to y} (\mathsf{Invertible}(f)
  \to \mathsf{BiInvertible}(f))\). Given \(I = (\inv
  f,f_R^i,f_L^i,f_R^{QI}, f_L^{QI}) : \mathsf{Invertible}(f)\) we can
  construct \(\mathsf{invToBiInv}(f)(I)\):
  \begin{itemize}
  \item \(\rinv f = \inv f\);
  \item \(\linv f = \inv f\);
  \item \(f_R = f_R^i\);
  \item \(f_L = f_L^i\);
  \item corecursively, we let \(f_R^{BI} =
    \mathsf{invToBiInv}(f_R)(f_R^{QI})\);
  \item similarly, let \(f_L^{BI} =
    \mathsf{invToBiInv}(f_L)(f_L^{QI})\).
  \end{itemize}
  Then \(\mathsf{invToBiInv}(f)(I) : \mathsf{BiInvertible}(f)\) as
  required.
\end{proof}

\noindent Next, one of the most important and complex theorems of the
paper is proved.

\begin{theorem}\label{bi-inv-composition}
  For any \(n \in \mathbb{N}\) and cells \(f,g\) where \(g \comp_n f\)
  is well-defined, given a bi-invertible structure on \(f\) and \(g\),
  we can generate a bi-invertible structure on \(g \comp_n f\).
\end{theorem}
\begin{proof}
  This will be proved by corecursion on the bi-invertible structure
  that is being generated. We will split into cases on \(n\).

  Suppose \(n = 0\). Let \((\rinv f, \linv f, f_R, f_L, f_R^{BI},
    f_L^{BI})\) and \((\rinv g, \linv g, g_R, g_L, g_R^{BI},
    g_L^{BI})\) be bi-invertible structures for \(f\) and \(g\)
    respectively. We define a bi-invertible structure on \(g
    \comp_0 f\) where:
    \begin{itemize}
    \item \(\rinv {(g \comp_0 f)} = \rinv f \comp_0 \rinv g\)
    \item \(\linv {(g \comp_0 f)} = \linv f \comp_0 \linv g\)
    \item \({(g \comp_0 f)}_R : g \comp_0 f \comp_0 \rinv f \comp_0
      \rinv g \to \id_z\) is the morphism
      \begin{equation*}
        a_{g, f, \rinv f \comp_0 \rinv g} \comp_0 (\id_g \comp_1 a_{f,\rinv f, \rinv g}^{-1}) \comp_0 (\id_g \comp_1 (f_R \comp_1 \id_{\rinv g})) \comp_0 (\id_g \comp_1 \lambda_{\rinv g}) \comp_0 g_R
      \end{equation*}
      This may be easier to understand from its string diagram, which
      is given below:
      \begin{center}
        \begin{tikzpicture}[every node/.style={rounded corners=5pt}]
          \node[draw] (GR) {\(g_R\)};
          \node[on grid, below=20pt of GR, draw] (FR) {\(f_R\)};
          \node[on grid, below left=30pt and 45pt of FR] (Base1) {\(g\)};
          \node[on grid, right=15pt of Base1] (Base2) {\(f\)};
          \node[on grid, right=60pt of Base2] (Base3) {\(\rinv f\)};
          \node[on grid, right=15pt of Base3] (Base4) {\(\rinv g\)};
          \draw (Base1) to ++(0,20pt) to [out=90,in=180] (GR);
          \draw (Base4) to ++(0,20pt) to [out=90,in=0] (GR);
          \draw (Base2) to ++(0,15pt) to [out=90,in=180] (FR);
          \draw (Base3) to ++(0,15pt) to [out=90,in=0] (FR);
        \end{tikzpicture}
      \end{center}
    \item \({(g \comp_0 f)}_L : \linv g \comp_0 \linv f \comp_0 f
      \comp_0 g \to \id_x\) is the morphism
      \begin{equation*}
        a_{\linv g, \linv f, f \comp_0 g} \comp_0 (\id_{\linv g} \comp_1 a_{})\comp_0 (\id_{\linv g} \comp_1 (f_L \comp_1 \id_g)) \comp_0 (\id_{\linv g} \comp_1 \lambda_g) \comp_0 g_L
      \end{equation*}
      which is given by the diagram:
      
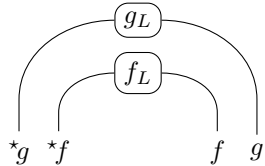
\begin{center}
        \begin{tikzpicture}
          \node[draw] (GL) {\(g_L\)}; \node[on grid, below=20pt of GL,
          draw] (FL) {\(f_L\)}; \node[on grid, below left=30pt and
          45pt of FL] (Base1) {\(\linv g\)}; \node[on grid,
          right=15pt of Base1] (Base2) {\(\linv f\)}; \node[on grid,
          right=60pt of Base2] (Base3) {\(f\)}; \node[on grid,
          right=15pt of Base3] (Base4) {\(g\)}; \draw (Base1) to
          ++(0,20pt) to [out=90,in=180] (GL); \draw (Base4) to
          ++(0,20pt) to [out=90,in=0] (GL); \draw (Base2) to
          ++(0,15pt) to [out=90,in=180] (FL); \draw (Base3) to
          ++(0,15pt) to [out=90,in=0] (FL);
        \end{tikzpicture}
      \end{center}
    \end{itemize}

    From the bi-invertible structures on \(f\) and \(g\),
    bi-invertible structures on \(f_L\), \(g_L\), \(f_R\), and \(g_R\)
    can be obtained. The associators and unitors used have
    bi-invertible structures by \cref{inv-to-bi-inv}
    and~\ref{inverse-invert}. All identity morphisms can be equipped
    with a bi-invertible structure using \cref{inv-to-bi-inv,identity}. Then by coinductive hypothesis we can generate
    bi-invertible structures on \({(g \comp_0 f)}_R\) and \({(g
      \comp_0 f)}_L\).

    Instead suppose \(n > 0\). Suppose we have \(\alpha: f \to g\) and \(\beta: h
    \to j\) where \(f \comp_{n-1} h\) and \(g \comp_{n-1} j\) are both
    well-defined. Further suppose:
    \begin{itemize}
    \item \((\rinv \alpha, \linv \alpha, \alpha_R, \alpha_L,
      \alpha_R^{BI}, \alpha_L^{BI})\) is a bi-invertible structure on
      \(\alpha\);
    \item \((\rinv \beta, \linv \beta, \beta_R, \beta_L, \beta_R^{BI},
      \beta_R^{BI})\) is a bi-invertible structure for \(\beta\).
    \end{itemize}
    Then we can define the following bi-invertible structure on
    \(\alpha \comp_n \beta\):
    \begin{itemize}
    \item \(\rinv {(\alpha \comp_n \beta)} = \rinv \alpha \comp_n
      \rinv \beta\)
    \item \(\linv {(\alpha \comp_n \beta)} = \linv \alpha \comp_n
      \linv \beta\)
    \item Let \({(\alpha \comp_n \beta)}_R\) be the following
      composition:
      \begin{equation*}
        (\alpha \comp_n \beta) \comp_0 (\rinv \alpha \comp_n \rinv \beta) \overset {i_{\alpha,\beta,\rinv \alpha, \rinv \beta}} \to (\alpha \comp_0 \rinv \alpha) \comp_n (\beta \comp_0 \rinv \beta) \overset {\alpha_R \comp_{n+1} \beta_R} \to \id_g \comp_n \id_j \to \id_{g \comp_{n-1} j}
      \end{equation*}
    \item Similarly, let \({(\alpha \comp_n \beta)}_L\) be the
      following composition:
      \begin{equation*}
        (\linv \alpha \comp_n \linv \beta) \comp_0 (\alpha \comp_n \beta) \overset {i_{\linv \alpha, \linv \beta, \alpha, \beta}} \to (\linv \alpha \comp_0 \alpha) \comp_n (\linv \beta \comp_0 \beta) \overset {\alpha _L \comp_{n+1} \beta _L} \to \id_f \comp_n \id_h \to \id_{f \comp_n h}
      \end{equation*}
    \end{itemize}
    Now, both \({(\alpha \comp_n \beta)}_R\) and \({(\alpha \comp_n
      \beta)}_L\) are the composition of cells with quasi-invertible
    structures given by the higher pre-category structure (and so have
    bi-invertible structures given by \Cref{inv-to-bi-inv}) and
    \(\alpha_L\), \(\alpha_R\), \(\beta_L\), and \(\beta_R\) which
    have structures \(\alpha_L^{BI}\), \(\alpha_R^{BI}\),
    \(\beta_L^{BI}\), and \(\beta_R^{BI}\). Therefore by coinductive
    hypothesis there are bi-invertible structures on both \(\rinv {(\alpha \comp_n \beta)}\) and \(\linv {(\alpha \comp_n \beta)}\).
\end{proof}

\begin{lemma}\label{inverses}
  Let \(f\) be a cell. Given a bi-invertible structure on \(f\), one
  can be generated on both \(\linv f\) and \(\rinv f\).
\end{lemma}
\begin{proof}
  A structure is given for \(\rinv f\) is bi-invertible and \(\linv
  f\) will follow by symmetry. As before let \((\rinv f, \linv f, f_R,
  f_L, f_R^{BI}, f_L^{BI})\) be a bi-invertible structure on \(f\).
  Then a bi-invertible structure on \(\rinv f\) can be generated by:
  \begin{itemize}
  \item \(\rinv {(\rinv f)} = f\)
  \item \(\linv {(\rinv f)} = f\)
  \item \({(\rinv f)}_R: \rinv f \comp_0 f \to \id\) is the morphism:
    \begin{equation*}
      \lambda_{\rinv f \comp_0 f}^{-1} \comp_0 (\rinv {f_L} \comp_1 \id_{\rinv f \comp_0 f}) \comp_0 a_{\linv f, f, \rinv f \comp_0 f} \comp_0 (\id_{\linv f} \comp_1 a_{f, \rinv f, f}^{-1}) \comp_0 (\id_{\linv f} \comp_1 (f_R \comp_1 \id_f)) \comp_0 (\id_{\linv f} \comp_1 \lambda_f) \comp_0 f_L
    \end{equation*}
    given by the string diagram:
    
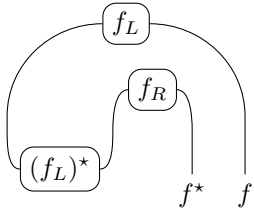
\begin{center}
      \begin{tikzpicture}
        \node (Finv) {\(\rinv f\)};
        \node [on grid, right=20pt of Finv] (F) {\(f\)};
        \node [on grid, above left=40pt and 15pt of Finv, draw] (FR) {\(f_R\)};
        \node [on grid, below left=30pt and 35pt of FR, draw] (FLInv) {\(\rinv {(f_L)}\)};
        \node [on grid, above right=55pt and 25pt of FLInv, draw] (FL) {\(f_L\)};
        \draw (Finv) to ++(0,25pt) to[out=90,in=0,looseness=0.8] (FR);
        \draw (FR) to[out=180,in=90,looseness=0.8] ++(-15pt,-15pt) to[out=-90,in=0,looseness=0.8] (FLInv);
        \draw (FLInv) to[out=180,in=-90,looseness=0.8] ++(-20pt,15pt) to ++(0pt,5pt) to[out=90,in=180] (FL);
        \draw (F) to ++(0,30pt) to[out=90,in=0] (FL);
      \end{tikzpicture}
    \end{center}
  \item \({(\rinv f)}_L: f \comp_0 \rinv f \to \id\) is given by
    \(f_R\)
  \end{itemize}
  Let \({(\rinv f)}_L\) be given \(f_R^{BI}\) as its bi-invertible
  structure. A bi-invertible structure on \({(\rinv f)}_R\) can be
  formed using \Cref{bi-inv-composition}, using that \(\rinv {f_L}\)
  has a bi-invertible structure given by coinductive hypothesis.
\end{proof}

\subsection{Half-Adjoint Inverses}\label{sec:hai}

Another type of equivalence is a half-adjoint equivalence. Whereas a
bi-invertible structure was a weakening of a quasi-invertible
structure, a half-adjoint invertible structure is a strict
strengthening of a quasi-invertible structure. This is done by adding
a coherence condition which effectively enforces \(f_L\) and \(f_R\)
to work ``nicely'' together. It turns out that there are two such
coherence conditions, known as snake equations or zigzag identities,
that can be added, yet each of these implies the other and so it
sufficient to provide one of these~\cite[Lemma
3.2]{nlab:adjoint_equivalence}. This is what gives rise to the name
\emph{half}-adjoint invertible.

\begin{definition}
  Given a globular set \(G\) with identities and composition, with an
  \(n\)-cell \(f : x \to y\), a \emph{half-adjoint invertible}
  structure on \(f\) is a
  tuple \((f', \alpha_f, \beta_f, \gamma_f, \alpha_f^{HAI},
  \beta_f^{HAI}, \gamma_f^{HAI})\) where:
  \begin{itemize}
  \item \(f'\) is an \(n\)-cell \(y \to x\);
  \item \(\alpha_f\) is an \((n+1)\)-cell \(f \comp_0 f' \to \id_y\);
  \item \(\beta_f\) is an \((n+1)\)-cell \(\id_x \to f' \comp_0 f\);
  \item \(\gamma_f\) is an \((n+2)\)-cell \((\rho_{f}^{-1} \comp_0
    (\id_f \comp_1 \beta_f) \comp_0 a_{f,f',f} \comp_0 (\alpha_f \comp_1 \id_f) \comp_0 \lambda_{f}) \to \id_{f}\);
  \item \(\alpha_f^{HAI}\) is a half-adjoint invertible structure on
    \(\alpha_f\);
  \item \(\beta_f^{HAI}\) is a half-adjoint invertible structure on
    \(\beta_f\);
  \item \(\gamma_f^{HAI}\) is a half-adjoint invertible structure on
    \(\gamma_f\).
  \end{itemize}

  Where \(\gamma_f\) can be graphically represented by the following
  diagram:
  
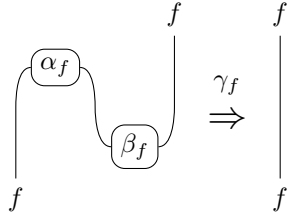
\begin{center}
    \begin{tikzpicture}
      \node (Bottom) {\(f\)};
      \node[on grid, above right=20pt and 45pt of Bottom, draw] (Cup){\(\beta_f\)};
      \node[on grid, above left=30pt and 30pt of Cup,draw] (Cap){\(\alpha_f\)};
      \node[on grid, above right=20pt and 45pt of Cap] (Top) {\(f\)};
      \draw (Bottom) to++(0,35pt) to[out=90,in=180,looseness=0.8] (Cap);
      \draw (Cup) to[out=180,in=-90,looseness=0.8] ++(-15pt,15pt) to[out=90,in=0,looseness=0.8] (Cap);
      \draw (Cup) to[out=0,in=-90,looseness=0.8] ++(15pt,15pt) to (Top);

      \node[on grid, below right=40pt and 20pt of Top,
      font=\fontsize{15}{24}\selectfont] (Eq) {\(\Rightarrow\)};
      \node[above=0cm of Eq] (Gamma)
      {\(\gamma_f\)};
      \node[on grid, right=40pt of Top] (L1)
      {\(f\)}; \node[on grid, below=70pt of L1] (L2) {\(f\)};
      \draw (L1) to (L2);
    \end{tikzpicture}
  \end{center}
\end{definition}

\begin{theorem}
  A half-adjoint invertible structure can be restricted to a
  quasi-invertible structure on the same morphism.
\end{theorem}
\begin{proof}
  Suppose \(f\) has half-adjoint invertible structure \((f',
  \alpha_f, \beta_f, \gamma_f, \alpha_f^{HAI}, \beta_f^{HAI},
  \gamma_f^{HAI})\). Then let:
  \begin{itemize}
  \item \(\inv f = {f'}\);
  \item \(f_R = \alpha_f\);
  \item \(f_L = \beta_f'\).
  \end{itemize}
  By coinduction, \(\alpha_f^{HAI}\) can be restricted to a
  quasi-invertible structure on \(f_R\) and \(\beta_f^{HAI}\) can be
  restricted to quasi-invertible structure on \(\beta_f\) which
  induces a structure on \(\inv {\beta_f}\) by \Cref{inverse-invert}.
\end{proof}

\begin{cor}
  A half-adjoint invertible structure can be restricted to a
  bi-invertible structure on the same morphism.
\end{cor}

\begin{theorem}\label{bi-inv-to-hai}
  A bi-invertible structure \((\rinv f, \linv f, f_R, f_L, \dots)\) on
  a cell \(f\) induces a half-adjoint invertible structure \((\rinv
  f, f_R, \dots)\) on \(f\).
\end{theorem}
\begin{proof}
  Let \((\rinv f, \linv f, f_R, f_L, f_R^{BI}, f_L^{BI})\) be a
  bi-invertible structure on \(f\). Then we give the right-adjoint
  invertible structure \((\rinv f, f_R, \beta_f, \gamma_f, f_R^{HAI},
  \beta_f^{HAI}, \gamma_f^{HAI})\) where:
  \begin{itemize}
  \item \(\beta_f\) is given by the following diagram:
    
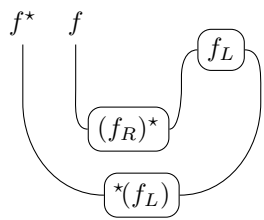
\begin{center}
      \begin{tikzpicture}
        \node (Finv) {\(f\)};
        \node [on grid, left=20pt of Finv] (F) {\({\rinv f}\)};
        \node [on grid, below right=40pt and 20pt of Finv, draw] (FRInv) {\(\rinv {(f_R)}\)};
        \node [on grid, above right=30pt and 35pt of FRInv, draw] (FL) {\(f_L\)};
        \node [on grid, below left=55pt and 30pt of FL, draw] (FLInv) {\(\linv {(f_L)}\)};
        \draw (Finv) to ++(0,-25pt) to[out=-90,in=180,looseness=0.9] (FRInv);
        \draw (FRInv) to[out=0,in=-90,looseness=0.6] ++(20pt,15pt) to ++(0,5pt) to[out=90,in=180] (FL);
        \draw (FL) to[out=0,in=90] ++(15pt,-15pt) to ++(0pt,-5pt) to[out=-90,in=0] (FLInv);
        \draw (F) to ++(0,-30pt) to[out=-90,in=180] (FLInv);
      \end{tikzpicture}
    \end{center}
  \item \(\gamma_f\) is given by the following diagram:
    
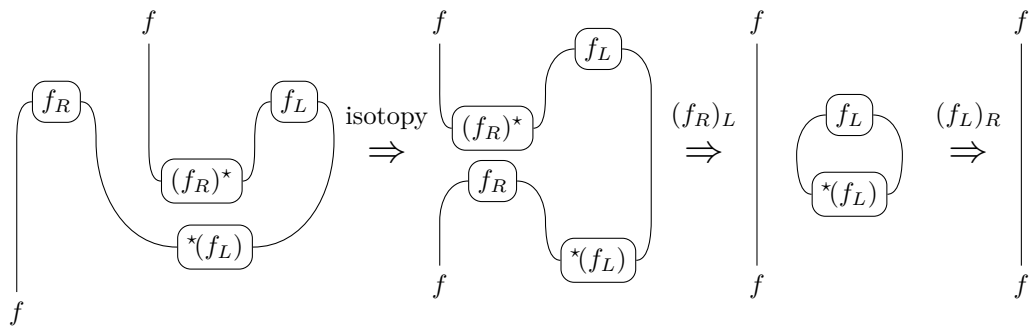
\begin{center}
      \begin{tikzpicture}
        \node (G) {\(f\)};
        \node [on grid, below left=30pt and 35pt of G, draw] (FR) {\(f_R\)};
        \node [on grid, below left=80pt and 15pt of FR] (G') {\(f\)};
        \node [on grid, below right=60pt and 20pt of G, draw] (FRInv) {\(\rinv {(f_R)}\)};
        \node [on grid, above right=30pt and 35pt of FRInv, draw] (FL) {\(f_L\)};
        \node [on grid, below left=55pt and 30pt of FL, draw] (FLInv) {\(\linv {(f_L)}\)};
        \draw (G) to ++(0,-45pt) to[out=-90,in=180,looseness=0.8] (FRInv);
        \draw (FRInv) to[out=0,in=-90,looseness=0.8] ++(20pt,15pt) to[out=90,in=180,looseness=0.9] (FL);
        \draw (FL) to[out=0,in=90,looseness=0.9] ++(15pt,-15pt) to[out=-90,in=0] (FLInv);
        \draw (FR) to[out=0,in=90,looseness=0.9] ++(15pt,-15pt) to[out=-90,in=180] (FLInv);
        \draw (FR) to[out=180,in=90,looseness=0.9] ++(-15pt,-15pt) to (G');

        \node [on grid, below right=50pt and 90pt of G,font=\fontsize{15}{24}\selectfont] (A1) {\(\Rightarrow\)};
        \node [above=0pt of A1] (A1Text) {isotopy};

        \node [on grid, right=110pt of G] (G1) {\(f\)};
        \node [on grid, below right=40pt and 20pt of G1, draw] (FRInv1) {\(\rinv {(f_R)}\)};
        \node [on grid, below=20pt of FRInv1, draw] (FR1) {\(f_R\)};
        \node [on grid, above right=30pt and 40pt of FRInv1, draw] (FL1) {\(f_L\)};
        \node [on grid, below right=30pt and 40pt of FR1, draw] (FLInv1) {\(\linv {(f_L)}\)};
        \node [on grid, below left=40pt and 20pt of FR1] (G'1) {\(f\)};
        \draw (G1) to ++(0,-25pt) to[out=-90,in=180,looseness=0.8] (FRInv1);
        \draw (FRInv1) to[out=0,in=-90,looseness=0.8] ++(20pt,15pt) to[out=90,in=180] (FL1);
        \draw (FL1) to[out=0,in=90] ++(20pt,-15pt) to ++(0,-50pt) to[out=-90,in=0,looseness=0.8] (FLInv1);
        \draw (FR1) to[out=0,in=90] ++(20pt,-15pt) to[out=-90,in=180,looseness=0.8] (FLInv1);
        \draw (FR1) to[out=180,in=90] ++(-20pt,-15pt) to (G'1);

        \node [on grid, below right=50pt and 100pt of G1,font=\fontsize{15}{24}\selectfont] (A2) {\(\Rightarrow\)};
        \node [above=0pt of A2] (A2Text) {\((f_R)_L\)};

        \node [on grid, right=120pt of G1] (G2) {\(f\)};
        \node [on grid, below=100pt of G2] (G'2) {\(f\)};
        \node [on grid, below right=35pt and 35pt of G2, draw] (FL2) {\(f_L\)};
        \node [on grid, below=30pt of FL2, draw] (FLInv2) {\(\linv {(f_L)}\)};
        \draw (G2) to (G'2);
        \draw (FL2) to[out=0,in=90] ++(20pt,-15pt) to[out=-90,in=0,looseness=0.8] (FLInv2);
        \draw (FL2) to[out=180,in=90] ++(-20pt,-15pt) to[out=-90,in=180,looseness=0.8] (FLInv2);

        \node [on grid, below right=50pt and 80pt of G2,font=\fontsize{15}{24}\selectfont] (A3) {\(\Rightarrow\)};
        \node [above=0pt of A3] (A3Text) {\((f_L)_R\)};

        \node [on grid, right=100pt of G2] (G3) {\(f\)};
        \node [on grid, below=100pt of G3] (G'3) {\(f\)};
        \draw (G3) to (G'3);
      \end{tikzpicture}

    \end{center}
  \item \(f_R^{HAI}\) can be generated by coinduction hypothesis on
    \(f_R^{BI}\).
  \item A bi-invertible structure can be formed for \(\beta_f\) as it
    is a composition of identities, morphisms with given bi-invertible
    structures and inverses of those morphisms. Then \(\beta_f^{HAI}\)
    can be formed coinductively.
  \item A bi-invertible structure can be put on \(\gamma_f\) as it is
    the composite of morphisms that have a given bi-invertible
    structure and an isotopy, which has a quasi-invertible structure
    given by \(G\) being a higher pre-category. Therefore, as before,
    \(\gamma_f^{HAI}\) can be formed by coinduction.
  \end{itemize}
  Hence, \((\rinv f, f_R, \beta_f, \gamma_f, f_R^{HAI},
  \beta_f^{HAI}, \gamma_f^{HAI})\) is in the form required.
\end{proof}

\begin{cor}\label{cor:equiv}
  Let \(G\) be a higher pre-category. Let \(n > 0\) and \(f\) be an
  \(n\)-cell of \(G\). Then the following are equivalent:
  \begin{itemize}
  \item \(f\) has a bi-invertible structure.
  \item \(f\) has a quasi-invertible structure.
  \item \(f\) has a half-adjoint invertible structure.
  \end{itemize}
\end{cor}
It should be stressed that this is an equivalence in that each
structure can be obtained from the others, and not that the various
transformations are in any way inverses to each other.

\section{Towards Contractibility}\label{sec:contractibility}

\Cref{sec:hai} (in particular \Cref{cor:equiv}) achieves one of the
goals set out in the introduction. The other property stated was that
bi-invertibility and half-adjoint invertibility should be contractible
types. In type theory, a type is contractible if there is an element
in that type to which all other terms of that type are propositionally
equal. Any contractible type is then equivalent to the unit type.
The natural analogue to the higher categorical setting would be to say
that a category is contractible if it is equivalent to the terminal
category.

\begin{definition}
  The \emph{terminal globular set} \(T\) is the globular set with exactly one
  \(n\)-cell for each \(n\). There is then only one choice for the
  source and target of each cell and all required equations hold by
  this uniqueness. We can further add identities and composition to
  this globular set. There is only one way these could be defined due
  to uniqueness.
\end{definition}

What this does not answer is what a suitable notion of equivalence
should be. As this could be dependent of the specific definition of a
higher category, we do not try to answer this here. Instead, we give a
notion of contractibility based on the work that has already been
done. Due to the simplicity of the terminal category it is highly
likely that this is a sufficient condition for contractibility.

\begin{definition}
  Let \(G\) be a globular set. \(G\)
  is \emph{contractible} if given any parallel cells \(f\) and \(g\)
  there is a cell \(f \to g\).
\end{definition}

In this section we aim to define a higher category of cells between
invertibility data. While the task of showing contractibility falls
beyond the scope of this paper, we do manage to prove partial
contractibility results for the bi-invertible case, and suggest how
this could be continued.

The definition of the cells between these types is largely inspired
by~\cite[Lemma 4.2.5]{hottbook}. This gives us a compositional way to
think about equivalence between these two types. It can be formulated
as follows: Suppose we have a cell \(f : x \to y\) and two
bi-invertible structures on it \((\rinv f, \linv f, f_R, f_L,
f_R^{BI}, f_L^{BI})\) and \((\rinv f{}' , \linv f{}', f_R', f_L',
f_R^{BI}{}', f_L^{BI}{}')\). A \emph{bi-invertible 1-morphism} is a tuple
consisting of the following data:
\begin{itemize}
\item \(\phi_R : \rinv f \to \rinv f{}'\);
\item \(\phi_R^{BI}\) is a bi-invertible structure on \(\phi_R\);
\item \(\psi_R : (\id_f \comp_1 \phi_R) \comp_0 f_R' \to f_R\);
\item \(\psi_R^{BI}\) is a bi-invertible structure on \(\psi_R\);
\item \(BIC_R\) is a bi-invertible 1-morphism from the induced bi-invertible structure on
  \((\id_f \comp_1 \phi_R) \comp_0 f_R'\) to \(f_R^{BI}\);
\item \(\phi_L\), \(\phi_L^{BI}\), \(\psi_L\), \(\psi_L^{BI}\), and
  \(BIC_L\) are similar but symmetric to above.
\end{itemize}

It should be noted that this is quite a natural way to define a cell
between these structures. We simply defined a cell for each part of
the structure separately, with pretty much only one way of defining
each such cell. The only oddity may be that we require these cells to
be bi-invertible. This is simply so that \((\id_f \comp_1 \phi_R)
\comp_0 f_R'\) has a canonical bi-invertible structure. Given this,
the construction can be continued. Suppose we have \((\phi_R, \psi_R,
 \dots)\) and \((\phi_R', \psi_R', \dots)\). Then a \emph{bi-invertible 2-morphism} can be defined as:
\begin{itemize}
\item a bi-invertible 1-morphism \(\epsilon: \phi_R \to \phi_R'\);
\item a bi-invertible 1-morphism from:
  
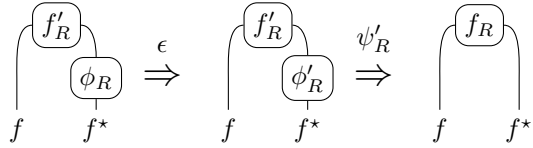
\begin{center}
    \begin{tikzpicture}
      \node (Base1) {\(f\)};
      \node [on grid, right=30pt of Base1] (Base2) {\(\rinv f\)};
      \node [on grid, above=20pt of Base2, draw] (G1) {\(\phi_R\)};
      \node [on grid, above left=20pt and 15pt of G1, draw] (Top1) {\(f_R'\)};
      \draw (Base1) to ++(0,20pt)[out=90,in=180,looseness=0.8] to (Top1);
      \draw (Base2) to (G1);
      \draw (G1) to[out=90, in=0,looseness=0.9] (Top1);

      \node [on grid, right=25pt of G1,font=\fontsize{15}{24}\selectfont](M1){\(\Rightarrow\)};
      \node [above=0pt of M1]{\(\epsilon\)};

      \node [on grid, below right=20pt and 25pt of M1](Base3) {\(f\)};
      \node [on grid, right=30pt of Base3] (Base4) {\(\rinv f\)};
      \node [on grid, above=20pt of Base4, draw] (G2) {\(\phi_R'\)};
      \node [on grid, above left=20pt and 15pt of G2, draw] (Top2) {\(f_R'\)};
      \draw (Base3) to ++(0,20pt)[out=90,in=180,looseness=0.8] to (Top2);
      \draw (Base4) to (G2);
      \draw (G2) to[out=90, in=0,looseness=0.9] (Top2);

      \node [on grid, right=25pt of G2,font=\fontsize{15}{24}\selectfont](M2){\(\Rightarrow\)};
      \node [above=0pt of M2] {\(\psi_R'\)};

      \node [on grid, below right=20pt and 25pt of M2](Base5) {\(f\)};
      \node [on grid, right=30pt of Base5] (Base6) {\(\rinv f\)};
      \node [on grid, above right=40pt and 15pt of Base5,draw] (Top3) {\(f_R\)};
      \draw (Base5) to ++(0cm,20pt)[out=90,in=180,looseness=0.8] to (Top3);
      \draw (Base6) to ++(0cm,20pt)[out=90,in=0,looseness=0.8] to (Top3);
    \end{tikzpicture}
  \end{center}
  to \(\psi_R\);
\item similar constructions for the remainder of the parts.
\end{itemize}

We conjecture that continuing in this way will generate a higher
category structure. Next is the main result of this section.

\begin{theorem}
  Given a cell \(f : x \to y\), there is a bi-invertible 1-morphism
  between any two bi-invertible structures on \(f\).
\end{theorem}

\begin{proof}
  Take structures \((\rinv f, \linv f, f_R, f_L, f_R^{BI}, f_L^{BI})\)
  and \((\rinv f{}' , \linv f{}', f_R', f_L', f_R^{BI}{}', f_L^{BI}{}')\).
  It will be sufficient to show that \(\phi_R\) and \(\psi_R\) can be
  constructed. Then the cell between bi-invertible structures can be
  obtained by coinductive hypothesis and the rest of the data follows
  from symmetry.

  Using \Cref{bi-inv-to-hai}, a half-adjoint invertible
  structure \((\rinv f{}', f_R', \beta_{f'}, \gamma_{f'}, {f_R'}^{HAI}, \beta_{f'}^{HAI}, \gamma_{f'}^{HAI})\) on \(f\) can be obtained.
  Then let \(\phi_R\) be the cell given by the diagram:
  
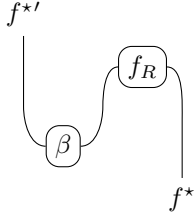
\begin{center}
    \begin{tikzpicture}
      \node (Bottom) {\(\rinv f\)};
      \node [on grid, above left=20pt and 45pt of Bottom,draw](Cup){\(\beta\)};
      \node [on grid, above right=30pt and 30pt of Cup, draw](Cap){\(f_R\)};
      \node [on grid, above left=20pt and 45pt of Cap](Top){\(\rinv f{}'\)};
      \draw (Bottom) to ++(0,30pt) to[out=90,in=0,looseness=0.8] (Cap);
      \draw (Cap) to[out=180,in=90,looseness=0.8] ++(-15pt,-15pt) to[out=-90,in=0,looseness=0.9] (Cup);
      \draw (Cup) to[out=180,in=-90,looseness=0.9] ++(-15pt,15pt) to (Top);
    \end{tikzpicture}
  \end{center}
  There is a clear bi-invertible structure on this morphism as it a
  composition of various bi-invertible cells. We now let
  \(\psi_R\) be:
  
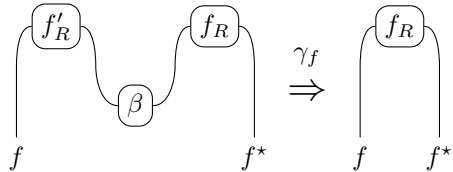
\begin{center}
    \begin{tikzpicture}
      \node (Bottom) {\(\rinv f\)};
      \node [on grid, above left=20pt and 45pt of Bottom,draw](Cup){\(\beta\)};
      \node [on grid, above right=30pt and 30pt of Cup, draw](Cap){\(f_R\)};
      \node [on grid, above left=30pt and 30pt of Cup,draw](Top){\(f_R'\)};
      \node [on grid, left=90pt of Bottom] (End) {\(f\)};
      \draw (Bottom) to ++(0,35pt) to[out=90,in=0,looseness=0.8] (Cap);
      \draw (Cap) to[out=180,in=90,looseness=0.8] ++(-15pt,-15pt) to[out=-90,in=0,looseness=0.9] (Cup);
      \draw (Cup) to[out=180,in=-90,looseness=0.9] ++(-15pt,15pt) to[out=90,in=0,looseness=0.8] (Top);
      \draw (Top) to[out=180,in=90,looseness=0.8] ++(-15pt,-15pt) to (End);

      \node [on grid, above right=25pt and 20pt of
      Bottom,font=\fontsize{15}{24}\selectfont] (M1)
      {\(\Rightarrow\)};
      \node [above=0pt of M1] {\(\gamma_f\)};

      \node [on grid, below right=25pt and 20pt of M1](Base5) {\(f\)};
      \node [on grid, right=30pt of Base5] (Base6) {\(\rinv f\)};
      \node [on grid, above right=50pt and 15pt of Base5,draw] (Top3) {\(f_R\)};
      \draw (Base5) to ++(0,35pt) to[out=90,in=180,looseness=0.8] (Top3);
      \draw (Base6) to ++(0cm,35pt) [out=90,in=0,looseness=0.8] to (Top3);
    \end{tikzpicture}
  \end{center}

  Which is bi-invertible by composition of bi-invertible structures
  \(\gamma_f\) and the bi-invertible structure on identities.
\end{proof}

This theorem effectively proves the first ``layer'' of
contractibility, that there is a cell between each pair of
\(0\)-cells. Whereas this layer used the adjoint coherences in its
proof, the next layer, of which the proof is omitted as it requires 3
dimensional reasoning that is not rigorous with the tools we have
here, can be constructed with swallowtail
equivalences~\cite{nlab:lax_2-adjunction}. These can be thought of as
the next coherence up. It is expected that the layer after this could
be nicely proved with the coherence condition after this one. We
conjecture that given a suitable notion of higher category, it should
be possible to show this structure is fully contractible.

One might ask why a similar proof will not work for invertible cells.
The reason for this is that the cells considered above are not the
canonical type of cell one would use to compare invertible cells. The
difference is that in the case of invertible cells, \(\phi_L\) and
\(\phi_R\) can be forced to be equal. In the above proof this is not
the case. In fact, given that invertibility is not a contractible type
in the type theory setting, we should expect that there is no such
proof.

It is expected that adding the structure of the half adjoint inverse
adds the exactly the kind of coherence we need to get the proof to
work. In the half-adjoint case we also have that \(\phi_L\) and
\(\phi_R\) must be equal. However a coherence for \(\gamma\) would also be necessary.
This higher dimensional cell makes it hard to work with the tools we
have, and so a proof is not attempted.

Another contractibility question that can be asked is the following:
Let \(\mathcal{C}\) be the infinity category freely generated by two
\(0\)-cells, a single \(1\)-cell between them and an invertibility
structure on this morphism. Is \(\mathcal{C}\) contractible?

We know that this is certainly not true for quasi-invertible
structures. Suppose we start with \(0\)-cells \(x\) and \(y\) and a
morphism \(f : x \to y\) and quasi-invertibility structure \((\inv f,
f_R, f_L, f_R^{QI}, f_L^{QI})\). Then the contractibility condition
effectively tells us that the cell required for a half-adjoint
invertibility structure should exist. However, this cell need not exist.
If we consider \(\mathbf{Cat}\) as a \(\omega\)-category (by
considering the 2-category and letting each higher globular set be either
empty or terminal), then we get that a quasi-invertible structure is
an equivalence and a half-adjoint invertible structure is an adjoint
equivalence. It is well known that not all equivalences are adjoint
equivalences.

However this problem does not arise in the bi-invertible case, and the
existence of the cell is enforced in the half-adjoint invertible case.
This gives good reason to believe that the infinity category freely
generated from these structures may be contractible though we have no
progress towards a proof of this. It is possible that similar
coinductive procedures can be employed to prove this.

\printbibliography{}
\end{document}